\documentclass[a4paper,11pt,reqno,pdftex]{amsart} %dvipdfmx %dvipdfmx
\usepackage{amsmath}
\usepackage{amssymb}
\usepackage{graphicx}

\newcommand{\Cov}{\mathop{\mathrm{Cov}}}

\theoremstyle{plain} %
\newtheorem{theorem}{Theorem}[section]

\newtheorem{lemma}[theorem]{Lemma}

\newtheorem{remark}[theorem]{Remark}

\makeatletter
    
    \@addtoreset{equation}{section}
  \makeatother

\title[A note on the elephant random walk with stops]{A note on the long time behavior of\\ the elephant random walk with stops}
\author{Tatsuya Akimoto}%\thanks{Date: \today}
\address{Graduate School of Engineering Science, Yokohama National University, 79-5 Tokiwadai, Hodogaya-ku, Yokohama, 240-8501, Kanagawa, Japan}
\author{Masato~Takei}%\thanks{Date: \today}
\address{Department of Applied Mathematics, Faculty of Engineering, Yokohama National University, 79-5 Tokiwadai, Hodogaya-ku, Yokohama, 240-8501, Kanagawa, Japan}
\email{takei-masato-fx@ynu.ac.jp}
\author{Keisuke Taniguchi}%\thanks{Date: \today}
\address{Graduate School of Engineering Science, Yokohama National University, 79-5 Tokiwadai, Hodogaya-ku, Yokohama, 240-8501, Kanagawa, Japan}
\usepackage{amsmath,amssymb}
\usepackage{color}

\begin{document}
\maketitle

\begin{abstract}
We study the long time behavior of the elephant random walk with stops, introduced by Kumar, Harbola and Lindenberg (2010), and
establish the phase transition of  the number of visited points up to time $n$, and the correlation between the position at time $n$ and the number of moves up to time $n$. %, as $n \to \infty$. coefficient 
\end{abstract}

%% main text
\section{Introduction}
\label{intro}

Recently there has been considerable interest in the elephant random walk (ERW), introduced by Sch\"{u}tz and Trimper \cite{SchutzTrimper04} (see Laulin \cite{Laulin22PhD} and references therein).
In this note we study the elephant random walk with stops (ERWS), which is a crucial generalization introduced by Kumar, Harbola and Lindenberg \cite{KumarHarbolaLindenberg10PRE}.
Let $s \in [0,1]$. Assume that $p,q \in [0,1]$ and $r \in [0,1)$ satisfy $p+q+r=1$. The first step $X_1$ of the walker is $+1$ with probability $s$, and $-1$ with probability $1-s$.
%\begin{align}
%\mbox{}
%X_1 = \begin{cases}
%+1 & \text{with probability $s$,} \\
%-1 & \text{with probability $1-s$.}
%\end{cases}
%\label{def:ERWfirststep}
%\end{align}
%$+1$ with probability $s$, and $-1$ with probability $1-s$.
For each $n \in \mathbb{N} := \{1,2,\ldots\}$, let $U_n$ be uniformly distributed on $\{1,\ldots,n\}$, and
\begin{align}
X_{n+1} &=  \begin{cases}
X_{U_n} &\mbox{with probability $p$,} \\
-X_{U_n} &\mbox{with probability $q$,} \\
0 &\mbox{with probability $r$.} \\
\end{cases}
\label{def:ERWwithStopsNextsteps}
\end{align} 
Each of choices in the above procedure is made independently. A one-dimensional random walk $\{S_n\}$ is defined by
$S_0:=0$, and $\displaystyle S_n= \sum_{i=1}^n X_i $ for $n\in \mathbb{N}$.
%\[ S_0:=0,\quad \mbox{and} \quad S_n= \sum_{i=1}^n X_i \quad \mbox{for $n\in \mathbb{N}$.} \]
The standard ERW corresponds to the special case $r=0$.

It turns out to be convenient to introduce the following new parameters:
$a:=p-q$, and $b:=1-r$.
%\begin{align*}
%a:=p-q,\quad \mbox{and} \quad b:=1-r.
%\end{align*}
Note that $b \in (0,1]$ and $a \in [-b,b]$.

An important quantity for the ERWS is the number $\Sigma_n$ of moves up to time $n$, i.e.,
\begin{align}
\Sigma_0:=0,\quad \mbox{and} \quad \Sigma_n := \sum_{i=1}^n 1_{\{X_i \neq 0\}} = \sum_{i=1}^n X_i^2 \quad \mbox{for $n\in \mathbb{N}$}.
\label{eq:ERWSDefSigma}
\end{align}
Let $\mathcal{F}_n$ be the $\sigma$-algebra generated by $X_1,\ldots,X_n$. Since $X_{n+1}^2 \in \{0,1\}$ and $P(X_{n+1}^2 = 1 \mid \mathcal{F}_n)  = (b\Sigma_n)/n$ for $n\in \mathbb{N}$,
%\begin{align*}
%P(X_{n+1}^2 = 1 \mid \mathcal{F}_n) 
%=1-P(X_{n+1}^2 = 0 \mid \mathcal{F}_n)= b \cdot \dfrac{\Sigma_n}{n}\quad \mbox{for $n\in \mathbb{N}$},
%\end{align*}
the process $\{\Sigma_n\}$ is nothing but the `laziest' minimal random walk model of elephant type with parameter $b$, %$1-r$,
studied in detail in Miyazaki and Takei \cite{MiyazakiTakei20JSP}. By Corollary 1 in \cite{MiyazakiTakei20JSP} (see also Lemma 2.1 in \cite{Bercu22}),
%It is known that 
\begin{align}
\lim_{n \to \infty} \dfrac{\Sigma_n}{n^b} = \Sigma>0 \quad \mbox{a.s.,} 
%\lim_{n \to \infty} \dfrac{\Sigma_n}{n^{1-r}} 
\label{eq:ERWSSigmanLimit}
\end{align}
where $\Sigma$ has a Mittag--Leffler distribution with parameter $b$. %$1-r$.
%The convergence \eqref{eq:ERWSSigmanLimit} was first shown in Theorem 3.1 of Gut and Stadtm\"{u}ller \cite{GutStadtmuller21SPL} using a martingale argument. %in 
%Bercu \cite{Bercu22} identified the distribution of $\Sigma$ using the moment method.  Thus the convergence result \eqref{eq:ERWSSigmanLimit} follows also from Corollary 1 of \cite{MiyazakiTakei20JSP}.

%By using the important fact that $\Sigma>0$ with probability one, 
In a remarkable paper \cite{Bercu22}, Bercu obtained several limit theorems for the ERWS $\{S_n\}$. Among other results in \cite{Bercu22}, we quote the central limit theorem and the law of the iterated logarithm.  %(CLT) (LIL)  
For a fixed $b \in (0,1]$, the critical value for $a \in [-b,b]$ is given by $a_c(b):=b/2$. 
%For a fixed $r \in [0,1)$, the critical value for $p \in [0,1-r]$ is given by $p_c(r):=\dfrac{3}{4} (1-r)$. 
 %$p_r := p/(1-r)$

\begin{theorem}[Bercu \cite{Bercu22}] \label{thm:Bercu22CLTLIL} Let $\phi(x):= \sqrt{2x\log \log x} $. \\
(i) (Subcritical regime)  If $a \in [-b,b/2)$ then $\dfrac{S_n}{\sqrt{\Sigma_n}}\stackrel{d}{\to}N\left(0,\dfrac{b}{b-2a}\right)$ as $n \to \infty$, and $\displaystyle \limsup_{n \to \infty} \pm \frac{S_n}{\phi(\Sigma_n)}=\sqrt{\dfrac{b}{b-2a}}$ a.s..\\
(ii) (Critical regime) If $a=b/2$ then $\dfrac{S_n}{\sqrt{\Sigma_n \log \Sigma_n}} \stackrel{d}{\to}N(0,1)$ as $n \to \infty$, and $\displaystyle \limsup_{n \to \infty} \pm \frac{S_n}{\phi(\Sigma_n \log \Sigma_n)}=1$ a.s.\\
(iii) (Supercritical regime) If $a \in (b/2,b]$ then there exists a random variable $L$ such that $P(L \neq 0)>0$ and $\displaystyle \lim_{n \to \infty}\frac{S_n}{n^a}=L$ a.s. 
\end{theorem}

\begin{remark} \label{rem:Remark!}
This theorem implies that $\{S_n\}$ is recurrent [resp. transient] when $a \in [-b, b/2]$ [resp. $a\in (b/2,b]$]:
More precisely, $S_n=0$ for infinitely many $n$ with probability one if $a \in [-b, b/2]$, while $S_n=0$ for only finitely many $n$ with positive probability if $a\in (b/2,b]$.
For the standard ERW (i.e. $b=1$), it is shown in \cite{GuerinLaulinRaschel23} that %the limit $L$ in Theorem \ref{thm:Bercu22CLTLIL} (iii) satisfies 
$P(L\neq 0)=1$.
\end{remark}

In this note, we carry out a further study on the phase transition concerning the long time behavior of the ERWS. Our first result in subsection \ref{ss:asrangeERWS} shows that the range of the ERWS admits phase transition. To our best knowledge, this is the first result on the range of all versions of elephant random walks. In subsection \ref{ss:CorrERWS}, we show that the correlation coefficient of $S_n$ and $\Sigma_n$
%between the position at time $n$ and the number of moves up to time $n$ 
exhibits a non-trivial phase transition when $r \in (0,1)$ i.e. the walker can stay on his own position with positive probability, and the number $\Sigma_n$ of moves up to time $n$ is random. 
Note that when $r=0$ the model is nothing but the original ERW, and the correlation coefficient of $S_n$ and $\Sigma_n \equiv n$ is zero for all $a \in [-1,1]$. Proofs are given in Sections \ref{sec:proofsA},   \ref{sec:proofsB1} and  \ref{sec:proofsB2}.

\section{Main results}

\subsection{Almost sure range of the elephant random walk with stops}
\label{ss:asrangeERWS}

For  the ERWS $\{S_n\}$, we define the {\it range} $\{R_n\}$ of by $R_0:=1$ and $R_n:=\#\{S_0,S_1,\ldots,S_n\}$ for $n \in \mathbb{N}$.  Combining the strong law of large numbers for $\{S_n\}$ obtained by \cite{Bercu22} and Theorem 2 in \cite{BoudabraMarkowsky21}, we have $\displaystyle \lim_{n \to \infty} \dfrac{R_n}{n}=0$ a.s. for any $b \in (0,1]$ and $a \in [-b,b]$.
The following theorem is a considerable improvement, which shows that the range admits phase transition at $a=a_c(b)=b/2$.

%In , i that the almost sure range of ERWS admits phase transition.
%To our best knowledge, this is the first result on the range of ERW.

\begin{theorem} \label{thm:asrangeERW} Assume that $b \in (0,1]$. %\[ R_0:=1 \quad \mbox{and} \quad R_n:=\#\{S_0,S_1,\ldots,S_n\} \quad \mbox{for $n \in \mathbb{N}$.} \]
%Let $\phi(x):= \sqrt{2x\log \log x} $. 
\begin{itemize}
\item[(i)] If $a \in [-b,b/2)$ then
$\displaystyle \sqrt{\dfrac{b}{b-2a}} \leq \limsup_{n \to \infty} \dfrac{R_n}{\phi(\Sigma_n)} \leq 2\sqrt{\dfrac{b}{b-2a}}$ a.s..
\item[(ii)] If $a=b/2$ then
$\displaystyle 1 \leq \limsup_{n \to \infty} \dfrac{R_n}{\phi(\Sigma_n\log \Sigma_n)} \leq 2$ a.s..
\item[(iii)] If $a \in (b/2,b]$, then 
$\displaystyle \lim_{n \to \infty} \dfrac{R_n}{n^a} = |L|$ a.s.,
where $L$ is the limiting random variable in Theorem \ref{thm:Bercu22CLTLIL} (iii). %\eqref{eq:Bercu22JSP(3.18)}.
\end{itemize}
\end{theorem}

\subsection{Correlation between the position and the number of moves}
\label{ss:CorrERWS}

In this subsection, we assume that $b=1-r \in (0,1)$, and thus the number of moves $\Sigma_n$ up to time $n$ is a nondegenerate random variable. %genuinely random. 
Our interest is how big the correlation between $S_n$ and $\Sigma_n$ is, and the next theorem implies that it is drastically changed at $a=a_c(b)=b/2$. 

The correlation coefficient of $X$ and $Y$ is denoted by $\rho[X,Y]$. The beta function is defined by
\begin{align*}
B(x,y) &:= \int_0^1 u^{x-1} (1-u)^{y-1}\,du \quad \mbox{for $x,y>0$.} 
\end{align*}
%{\color{blue} Recall that the beta function and the gamma function are defined by
%\begin{align*}
%B(x,y) &:= \int_0^1 u^{x-1} (1-u)^{y-1}\,du \quad \mbox{for $x,y>0$,} \\
%\Gamma(x) &:= \int_0^{\infty} u^{x-1}e^{-u}\,du  \quad \mbox{for $x>0$}
%\end{align*}
%and satisfy that $B(x,y)=\dfrac{\Gamma(x)\Gamma(y)}{\Gamma(x+y)}$.
%}

\begin{theorem} \label{thm:Taniguchi23ThmA} Assume that $b \in (0,1)$. If $a=0$ or $s=1/2$, then $\rho [S_n,\Sigma_n]=0$ for all $n \in \mathbb{N}$.
%\begin{align} %\Cov
%\rho [S_n,\Sigma_n]=0
%\quad \mbox{for all $n \in \mathbb{N}$.}
%\label{eq::Taniguchi23ThmA_1}
%\end{align}
If $a \neq 0$ and $s \neq 1/2$, then as $n \to \infty$, 
\begin{align*}
\rho[S_n,\Sigma_n]\sim 
\begin{cases}
(2s-1) \cdot  P_{a,b}n^{-(b-2a)/2}&\mbox{if $a \in [-b,b/2)$,} \\
(2s-1) \cdot  Q_a (\log n)^{-1/2}&\mbox{if $a=b/2$,}\\
(2s-1) \cdot R_{a,b,s} &\mbox{if $a \in (b/2,b]$.}\\
\end{cases}
\end{align*}
%For $a \in [-b,b/2)$, the constant $P_{a,b}$ is defined by
%\begin{align*}
%P_{a,b} &:= \dfrac{\sqrt{(b-2a)\Gamma(b)} \cdot \{ (a+b) \cdot B(a+1,b)-1 \}}{\Gamma(a+1) \cdot \sqrt{b \cdot B(b,b)-1}},
%\end{align*}
%which is positive [resp. negative] for $a \in (0,b/2)$ [resp. $a \in [-b,0)$].
Here $P_{a,b}$, $Q_a$ and $R_{a,b,s}$ are constants given by
\begin{align*}
P_{a,b} &= \dfrac{\sqrt{(b-2a)\Gamma(b)} \cdot \{ (a+b) \cdot B(a+1,b)-1 \}}{\Gamma(a+1) \cdot \sqrt{b \cdot B(b,b)-1}}, \\
Q_a &= \dfrac{\sqrt{\Gamma(2a)} \cdot \{ a \cdot B(a,2a)-1 \}}{\Gamma(a+1) \cdot \sqrt{2a \cdot B(2a,2a)-1}}, \\
R_{a,b,s}&= \dfrac{\sqrt{2a-b} \cdot \{ a \cdot B(a,b)-1 \}}{\sqrt{\{a^2 \cdot B(a,a)- (2s-1)^2 \cdot (2a-b)\}\{b \cdot B(b,b)-1\}}}.
\end{align*}
Note that $P_{a,b}$ is positive [resp. negative] for $a \in (0,b/2)$ [resp. $a \in [-b,0)$] while $Q_a$ and $R_{a,b,s}$ are positive.
\end{theorem}

By Theorem \ref{thm:Taniguchi23ThmA}, the correlation between $S_n$ and $\Sigma_n$ vanishes as $n \to \infty$ when $a \in [-b, b/2]$, while it remains positive when $a\in (b/2,b]$. This is in line with our expectations, in view of the recurrence property of the ERWS, described in Remark \ref{rem:Remark!}. The (sub)critical ERWS is recurrent and after having completed an excursion (that is when $S_n = 0$) the walk has the same number of $+1$ and $-1$ steps in its history. Hence the probability of moving to $+1$ is the same as going to $-1$ independently of the past, that is, the signs of the excursions are determined by independent coin flips. It shows that in these regimes, the random sign of $S_n$ causes cancellation in a way which makes $S_n$ asymptotically independent of $\Sigma_n$. In the supercritical regime, the walk is transient and after the last return to $0$, the more non-zero steps the walk takes the further it gets away from $0$.

On the other hand, the next theorem says that $S_n^2$ and $\Sigma_n$ has %asymptotically 
a positive correlation for all $a \in [-b,b]$, which means that conditioning on $\Sigma_n$ gives non-trivial information on $S_n^2$ and hence on $|S_n|$.
%as in the next theorem. 
%One of our future problem is to obtain an intuitive explanation for it.

\begin{theorem} \label{thm:Taniguchi23ThmB} For $b\in (0,1)$,\begin{align}
\lim_{n \to \infty} \rho[S_n^2,\Sigma_n] =
\begin{cases}
Q'_{b} &\mbox{if $a \in [-b,b/2]$,} \\
R'_{a,b} &\mbox{if $a \in (b/2,b]$.} 
\end{cases}
\label{eq::Taniguchi23ThmB_2}
\end{align}
Here $Q'_b$ and $R'_{a,b}$ are positive constants given by
\begin{align*}
Q'_b= \sqrt{\dfrac{b \cdot B(b,b)-1}{3b \cdot B(b,b)-1}}, \
R'_{a,b}=\dfrac{2a \cdot B(2a,b)-1}{\sqrt{\{\frac{6(2a^2+2ab-b^2)}{4a-b} \cdot B(2a,2a)-1\}\{b \cdot B(b,b)-1\}}}.
\end{align*}
\end{theorem}

The limit in \eqref{eq::Taniguchi23ThmB_2} can be regard as a function of $a \in [-b,b]$ for fixed $b \in (0,1)$: See Figure \ref{graph:ATTGraph}.

\begin{figure}
\begin{tabular}{c}
\includegraphics[keepaspectratio,scale=0.40]{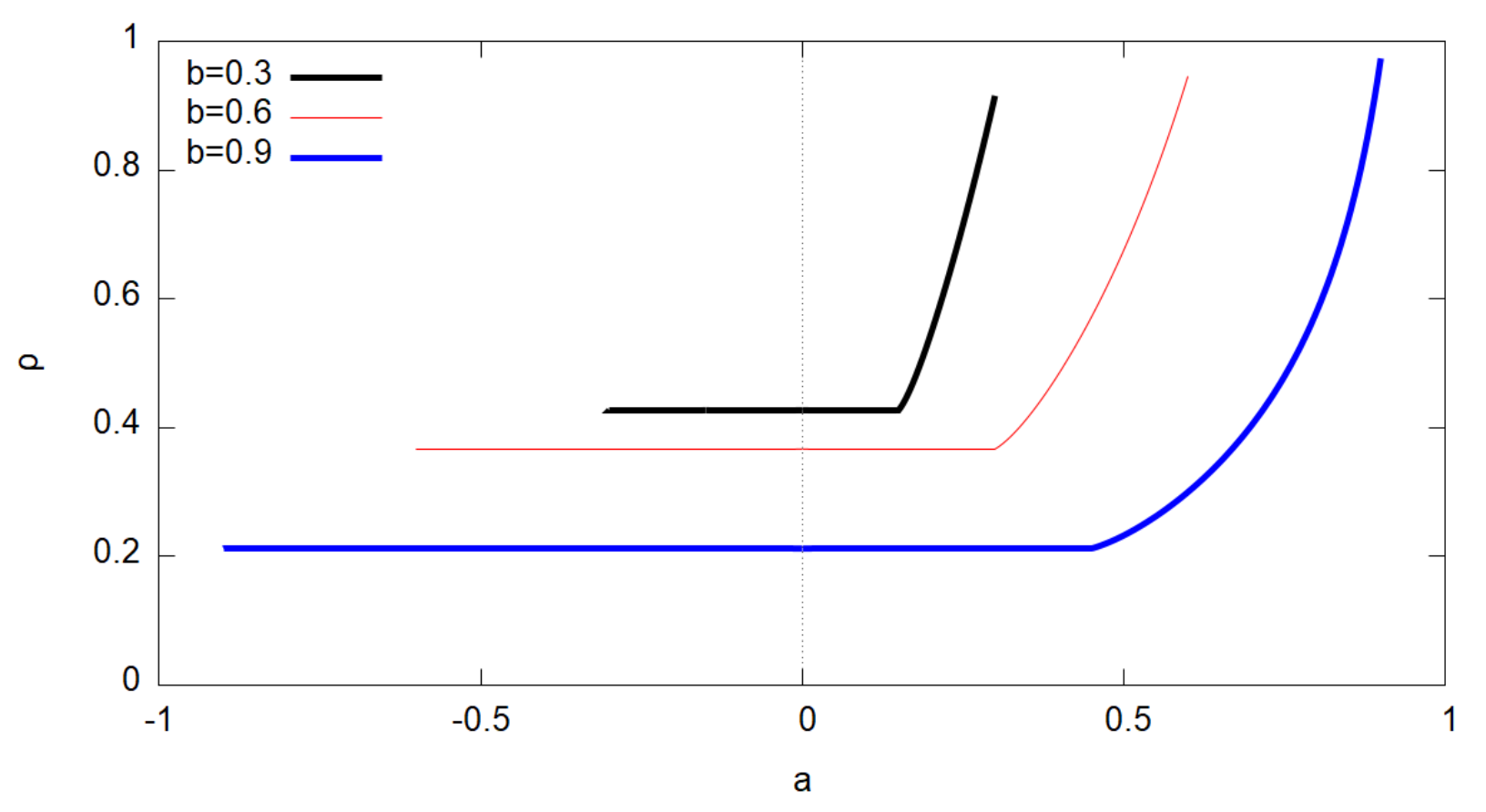}
\end{tabular}
\caption{The graphs of $\displaystyle \rho=\lim_{n \to \infty} \rho[S_n^2,\Sigma_n]$ as a function of $a \in [-b,b]$ for $b=0.3$, $0.6$ and $0.9$. Note that the critical value is $a_c(b)=b/2$.}
\label{graph:ATTGraph}
\end{figure}

\section{Proof of Theorem \ref{thm:asrangeERW}}

\label{sec:proofsA}

Theorem \ref{thm:asrangeERW} is a consequence of Theorem \ref{thm:Bercu22CLTLIL} and the following lemma, which is an extension of Theorem 2 in Boudabra and Markowsky \cite{BoudabraMarkowsky21}, and is of independent interest. 
%, together with \eqref{eq:Bercu22JSP(3.4)}, \eqref{eq:Bercu22JSP(3.12)}, and \eqref{eq:Bercu22JSP(3.18)}.

%210610 / 220805 revival
\begin{lemma} \label{lem:BoudabraMarkowsky21extended} Let $\{x_n\}$ be an integer valued sequence with $x_n-x_{n-1} \in \{-1,0,+1\}$ for each $n \in \mathbb{N}$.
%\[ \mbox{$x_n-x_{n-1} \in \{-1,0,+1\}$ for each $n \in \mathbb{N}$.} \]
We define $\{r_n\}$ by 
$r_0:=1$ and $r_n:=\#\{x_0,x_1,\ldots,x_n\}$ for $n \in \mathbb{N}$.
%\[ r_0:=1 \quad \mbox{and} \quad r_n:=\#\{x_0,x_1,\ldots,x_n\} \quad \mbox{for $n \in \mathbb{N}$.} \]
Let $f(n)$ be a positive real valued function defined for all sufficiently large $n \in \mathbb{Z}_+:=\{0,1,2,\ldots\}$, which is increasing and satisfies that $\displaystyle \lim_{n \to \infty} f(n)=+\infty$.
\begin{itemize}
\item[(i)] If $\displaystyle \lim_{n \to \infty} \dfrac{x_n}{f(n)}=c$, then
$\displaystyle \lim_{n \to \infty} \dfrac{r_n}{f(n)}=|c|$.
\item[(ii)] If $\displaystyle \limsup_{n \to \infty} \pm \dfrac{x_n}{f(n)}=1$, 
%\[ 
%\displaystyle \limsup_{n \to \infty} \dfrac{x_n}{f(n)}=1,\quad  \liminf_{n \to \infty} \dfrac{x_n}{f(n)}=-1
%\]
then $\displaystyle 1 \leq \limsup_{n \to \infty} \dfrac{r_n}{f(n)} \leq 2$.
\end{itemize}
\end{lemma}

\begin{proof} Without loss of generality, we can assume that $x_0=0$. 

\noindent (i) Suppose that $c>0$. For any $\varepsilon \in (0,c)$, there exists $N$ s.t. $(0<) (c-\varepsilon) f(n) \leq x_n \leq (c+\varepsilon)f(n)$ for all $n \geq N$.
%\[ (0<) (c-\varepsilon) f(n) \leq x_n \leq (c+\varepsilon)f(n) \quad \mbox{for all $n \geq N$.} \]
If $n \geq N$ then $\{x_0,x_1,\ldots,x_n\} \supset \{0,\ldots,\lfloor (c-\varepsilon) f(n)\rfloor \}$ 
%\[ \{x_0,x_1,\ldots,x_n\} \supset \{0,\ldots,\lfloor (c-\varepsilon) f(n)\rfloor \}, \]
and
$\{ x_N,\ldots,x_n\} \subset [0,(c+\varepsilon) f(n)]$,
%\[ \{ x_N,\ldots,x_n\} \subset [0,(c+\varepsilon) f(n)], \]
which imply that
$\lfloor (c-\varepsilon) f(n)\rfloor \leq r_n \leq r_{N-1} + (c+\varepsilon) f(n)$.
%\[ r_n \geq \lfloor (c-\varepsilon) f(n)\rfloor \]
%and
%\[ r_n \leq r_{N-1} + (c+\varepsilon) f(n). \]
Thus we have $\displaystyle c-\varepsilon \leq \liminf_{n \to \infty} \dfrac{r_n}{f(n)} \leq \limsup_{n \to \infty} \dfrac{r_n}{f(n)} \leq c+\varepsilon$.
%\[ c-\varepsilon \leq \liminf_{n \to \infty} \dfrac{r_n}{f(n)} \leq \limsup_{n \to \infty} \dfrac{r_n}{f(n)} \leq c+\varepsilon. \]
Letting $\varepsilon \searrow 0$, we obtain the desired conclusion when $c>0$. For the case $c<0$, consider $\{-x_n\}$. Next suppose that $c=0$. For any $\varepsilon>0$, there exists $N$ such that $-\varepsilon f(n) \leq x_n \leq \varepsilon f(n)$ for all $n \geq N$.
Since $ \{ x_N,\ldots,x_n\} \subset [-\varepsilon f(n),\varepsilon f(n)]$ for all $n \geq N$,
%\[ \{ x_N,\ldots,x_n\} \subset [-\varepsilon f(n),\varepsilon f(n)] \quad \mbox{for all $n \geq N$,} \]
we have $\displaystyle 0 \leq \liminf_{n \to \infty} \dfrac{r_n}{f(n)} \leq \limsup_{n \to \infty} \dfrac{r_n}{f(n)} \leq 2\varepsilon$.
%\[ 0 \leq \liminf_{n \to \infty} \dfrac{r_n}{f(n)} \leq \limsup_{n \to \infty} \dfrac{r_n}{f(n)} \leq 2\varepsilon. \]
This completes the proof of (i).

\noindent (ii) For any $\varepsilon>0$, there exists $N$ such that $-(1+\varepsilon)f(n)\leq x_n \leq (1+\varepsilon)f(n)$ for all $n \geq N$.
%\[ -(1+\varepsilon)f(n)\leq x_n \leq (1+\varepsilon)f(n) \quad \mbox{for all $n \geq N$.} \]
Noting that 
$\{ x_N,\ldots,x_n\} \subset [-(1+\varepsilon)f(n),(1+\varepsilon) f(n)]$ for all $n \geq N$, we have 
$r_n \leq r_{N-1} + 2(1+\varepsilon) f(n)$ for all $n \geq N$,
%\[ r_n \leq r_{N-1} + 2(1+\varepsilon) f(n) \quad \mbox{for all $n \geq N$,} \]
and $\displaystyle \limsup_{n \to \infty} \dfrac{r_n}{f(n)} \leq 2(1+\varepsilon)$.
%\[ \limsup_{n \to \infty} \dfrac{r_n}{f(n)} \leq 2(1+\varepsilon). \]
On the other hand, for any $\varepsilon \in (0,1)$, $x_n \geq (1-\varepsilon) f(n)(>0)$ for infinitely many $n$.
%\[ x_n \geq (1-\varepsilon) f(n)(>0) \quad \mbox{for infinitely many $n$.} \]
As $\{x_0,x_1,\ldots,x_n\} \supset \{0,\ldots,\lfloor (1-\varepsilon) f(n)\rfloor \}$ for infinitely many $n$,
we have $r_n \geq \lfloor (1-\varepsilon) f(n)\rfloor$ for infinitely many $n$,
%\[ r_n \geq \lfloor (1-\varepsilon) f(n)\rfloor \quad \mbox{for infinitely many $n$,}  \]
and $\displaystyle \limsup_{n \to \infty}\dfrac{r_n}{f(n)} \geq 1-\varepsilon$.
%\[ \limsup_{n \to \infty} \dfrac{r_n}{f(n)} \geq 1-\varepsilon. \]
This completes the proof of (ii).
\end{proof}

%\begin{proof}[Proof of Theorem \ref{thm:asrangeERW}] (i) [resp. (ii)] is a consequence of Lemma \ref{lem:BoudabraMarkowsky21extended} (ii) and Eq. \eqref{eq:Bercu22JSP(3.4)} [resp. Eq. \eqref{eq:Bercu22JSP(3.12)}]. (iii) follows from Lemma \ref{lem:BoudabraMarkowsky21extended} (i) and Eq. \eqref{eq:Bercu22JSP(3.18)}.
%\end{proof}

\section{Proof of Theorem \ref{thm:Taniguchi23ThmA}}

\label{sec:proofsB1}

For $x>-1$ and $n \in \mathbb{N}$, we define
\begin{align}
c_n(x):=\prod_{k=1}^{n-1} \left(1+\dfrac{x}{k}\right) = \dfrac{\Gamma(n+x)}{\Gamma(n)\Gamma(x+1)}
\sim \dfrac{n^x}{\Gamma(x+1)}\quad \mbox{as $n \to \infty$,}
\label{def+asymp:c_n(x)}
%\label{def:c_n(x)}
\end{align}
where we used Stirling's formula.

%\begin{align}
%c_n(x):=\prod_{k=1}^{n-1} \left(1+\dfrac{x}{k}\right) = \dfrac{\Gamma(n+x)}{\Gamma(n)\Gamma(x+1)}.
%\label{def:c_n(x)}
%\end{align}
%By Stirling's formula,
%\begin{align}
%c_n(x) \sim \dfrac{n^x}{\Gamma(x+1)}\quad \mbox{as $n \to \infty$.}
%\label{asymp:c_n(x)}
%\end{align}

%Recall that $a=p-q$ qnd $b=p+q=1-r$.

%Proof of Theorem \ref{thm:Taniguchi23ThmA}

\begin{lemma} Assume that $b \in (0,1)$ and $a \in [-b,b]$. If $a=0$ or $s=1/2$, then $\Cov[S_n,\Sigma_n] = 0$ for all $n \in \mathbb{N}$.
%\begin{align}
%\Cov[S_n,\Sigma_n] = 0 \quad \mbox{for all $n \in \mathbb{N}$.}
%\label{eq:Lemma4.1a}
%\end{align}
If $a \neq 0$ and $s \neq 1/2$, then
\begin{align*}
\Cov[S_n,\Sigma_n]\sim \dfrac{2s-1}{\Gamma(a+1)\Gamma(b+1)} \cdot \{ (a+b) \cdot B(a+1,b) -1 \} \cdot n^{a+b}
%Original:
%\frac{2s-1}{b\Gamma(a+b)} \cdot \left\{ 1-\frac{\Gamma(a+b)}{a\Gamma(a)\Gamma(b)} \right\} \cdot n^{a+b}
\quad\mbox{as $n\rightarrow\infty$.}
%\label{eq::Taniguchi23ThmA_2}
\end{align*}
\end{lemma}

\begin{proof}
Eqs. (C.17), (4.20) and (C.19) in \cite{Bercu22} can be read as follows:
\begin{align}
 E[S_n] &= (2s-1) \cdot c_n(a),
 %\dfrac{\Gamma(n+a)}{\Gamma(n)\Gamma(a+1)} \cdot E[X_1], 
 \label{eq:Bercu22JSP(C.17)} \\
 E[\Sigma_n] &= c_n(b),
 % \dfrac{\Gamma(n+b)}{\Gamma(n)\Gamma(b+1)}.
 \label{eq:Bercu22JSP(4.20)m=1} \\
 E[S_n\Sigma_n] &= \dfrac{2s-1}{b} \cdot \{(a+b) \cdot c_n(a+b)-a \cdot c_n(a)\}.
 \label{eq:Bercu22JSP(C.19)} 
\end{align}
If $a=0$ then $E[S_n\Sigma_n]  = (2s-1) \cdot c_n(b) = E[S_n] E[\Sigma_n]$ for all $n \in \mathbb{N}$.
%\begin{align*}
%E[S_n\Sigma_n] &= \dfrac{2s-1}{b} \cdot b \cdot  c_n(b) \\
%&= (2s-1) \cdot c_n(b) = E[S_n] E[\Sigma_n]\quad \mbox{for all $n \in \mathbb{N}$.}
%\end{align*}
If $s=1/2$ then $E[S_n\Sigma_n]= 0=E[S_n] E[\Sigma_n] $ for all $n \in \mathbb{N}$. %Thus we have \eqref{eq:Lemma4.1a}.
%As for \eqref{eq::Taniguchi23ThmA_2}, 
If $a \neq 0$ and $s \neq 1/2$, then by \eqref{def+asymp:c_n(x)},
\begin{align*}
E[S_n\Sigma_n] %&\sim \dfrac{2s-1}{b} \cdot (a+b) \cdot c_n(a+b) 
\sim \dfrac{2s-1}{b\Gamma(a+b)} \cdot n^{a+b}%, 
\quad \mbox{and} \quad %\\
%\intertext{and}
E[S_n]E[\Sigma_n] &\sim %(2s-1) \cdot \dfrac{n^a}{\Gamma(a+1)} \cdot \dfrac{n^b}{\Gamma(b+1)} =
 \dfrac{2s-1}{\Gamma(a+1)\Gamma(b+1)} \cdot n^{a+b},
\end{align*}
as $n \to \infty$. This completes the proof.
\end{proof}

\begin{lemma} Assume that $b \in (0,1]$. As $n \to \infty$,
\begin{align*}
V[S_n] &\sim 
\begin{cases}
\dfrac{1}{(b-2a)\Gamma(b)} \cdot n^b  &\mbox{if $a\in [-b,b/2)$,} \\
\dfrac{1}{\Gamma(2a)} \cdot n^{2a}\log n&\mbox{if $a=b/2$,} \\
\dfrac{a^2 \cdot B(a,a)-(2s-1)^2 \cdot (2a-b)}{(2a-b)\Gamma(a+1)^2} \cdot n^{2a} 
%Original:
%\left\{ \dfrac{1}{(2a-b)\Gamma(2a)}-\dfrac{(2s-1)^2}{a^2\Gamma(a)^2} \right\} \cdot n^{2a} 
& \mbox{if $a\in (b/2,b]$.}
\end{cases}
\end{align*}
\end{lemma}

\begin{proof}
%It is shown around Eqs. (A.5) and (A.6) in \cite{Bercu22} that  
From the calculation leading Eq. (A.6) in \cite{Bercu22},% (see also Lemma \ref{lem:Bercu18LemmaVariant} below),
\begin{align}
E[S_n^2] %&= c_n(2a)\cdot \left\{ 1+ b  \cdot \sum_{k=1}^{n-1} \dfrac{c_k(b)}{k \cdot c_{k+1} (2a)}\right\} \notag \\ % see the 2nd line of E[S_n^2]
&=\begin{cases}
\dfrac{b \cdot c_n(b) - 2a \cdot c_n(2a)}{b-2a} &\mbox{if $a \neq b/2$,} \\
\displaystyle 2a \cdot c_n(2a) \cdot \sum_{k=1}^n \dfrac{1}{k+2a}\ &\mbox{if $a=b/2$.} 
% c_n(2a)\cdot \left( 1+ 2a  \cdot \sum_{k=1}^{n-1} \dfrac{1}{k+2a}\right) 
\label{eq:Bercu22(A.6)}
%Bercu (2022):
%2a \cdot c_n(2a) \cdot \sum_{k=1}^n \dfrac{1}{k+2a}
%b-verion:
%b \cdot c_n(b) \cdot \sum_{k=1}^n \dfrac{\Gamma(k+b-1)}{\Gamma(k+b)} \\  
%&\sim b \cdot \dfrac{n^b}{\Gamma(b+1)} \cdot \log n = \dfrac{n^b \log n}{\Gamma(b)} \quad \mbox{as $n \to \infty$.}
\end{cases}
\end{align}
By \eqref{def+asymp:c_n(x)}, we have as $n \to \infty$,
\begin{align}
E[S_n^2]  
\sim \begin{cases}
%\dfrac{b}{b-2a} \cdot \dfrac{n^b}{\Gamma(b+1)}=
\dfrac{1}{(b-2a)\Gamma(b)} \cdot n^b &\mbox{if $a \in [-b,b/2)$,} \\
%2a \cdot \dfrac{n^{2a}}{\Gamma(2a+1)} \cdot \log n=
\dfrac{1}{\Gamma(2a)} \cdot n^{2a}\log n &\mbox{if $a=b/2$.} \\
%-\dfrac{2a}{b-2a} \cdot \dfrac{n^{2a}}{\Gamma(2a+1)} =
\dfrac{1}{(2a-b)\Gamma(2a)} \cdot n^{2a}&\mbox{if $a\in (b/2,b]$.}
\end{cases}
\label{eq:Bercu22JSP(A.6)VarAsymp}
\end{align}
%We do not use:
%From Eq. (3.22) in \cite{Bercu22}, if $2a>b$, then
%\begin{align*}
%\lim_{n \to \infty} E\left[ \left(\dfrac{S_n}{n^a}\right)^2\right] = \dfrac{1}{(2a-b)\Gamma(2a)}.
%\label{eq:Bercu22JSP(3.22)}
%\end{align*}
On the other hand, by \eqref{eq:Bercu22JSP(C.17)} and \eqref{def+asymp:c_n(x)}, $E[S_n]^2 \sim (2s-1)^2n^{2a} /\{\Gamma(a+1)^2\}$ as $n \to \infty$.
%\begin{align*}
%E[S_n]^2 \sim (2s-1)^2 \cdot \left\{ \dfrac{n^a}{\Gamma(a+1)} \right\}^2 = \dfrac{(2s-1)^2}{a^2 \Gamma(a)^2} \cdot n^{2a} \quad \mbox{as $n \to \infty$.}
%\end{align*}
Now the desired conclusion follows.
\end{proof}

\begin{lemma} \label{lem:VSigman}
If $b \in (0,1)$ then $V[\Sigma_n] \sim \dfrac{b \cdot B(b,b) -1}{\Gamma(b+1)^2}  \cdot n^{2b}$ as $n \to \infty$.
%\begin{align*}
%V[\Sigma_n] &\sim \dfrac{b \cdot B(b,b) -1}{b^2\Gamma(b)^2}  \cdot n^{2b}
%Original:
%\left\{\dfrac{2}{\Gamma(2b+1)}- \dfrac{1}{\Gamma(b+1)^2}\right\}\cdot n^{2b}  
%\quad \mbox{as $n \to \infty$.}
%= \left\{1-\dfrac{\Gamma(2b)}{b^2 \Gamma(b)^2} \right\} \cdot\dfrac{n^{2b}}{\Gamma(2b)}
%\end{align*}
\end{lemma}

\begin{proof}
From Eq. (4.20) in \cite{Bercu22}, 
\begin{align}
%E[\Sigma_n] = c_n(b)
%\quad \mbox{and} \quad 
E[\Sigma_n(\Sigma_n+1)] = 2 c_n(2b)
\quad \mbox{for $n \in \mathbb{N}$.}
\label{eq:Bercu22(4.20)1st2nd}
\end{align}
We have $V[\Sigma_n] = E[\Sigma_n^2] - E[\Sigma_n]^2 =\{2 c_n(2b) - c_n(b)\} - \{c_n(b)\}^2$,
%\begin{align*}
%V[\Sigma_n] &= E[\Sigma_n^2] - E[\Sigma_n]^2 =\{2 c_n(2b) - c_n(b)\} - \{c_n(b)\}^2,
%\end{align*}
which together with \eqref{def+asymp:c_n(x)} yields the claim.
\end{proof}

Combining the above lemmata, we have Theorem \ref{thm:Taniguchi23ThmA}.
For $a \in [-b,0)$, since $a(b-1) > 0$, we have $(a+b) \cdot B(a+1,b) \leq (a+b)/\{(a+1)b\} < 1$, and thus $P_{a,b}<0$. As for $a \in (0,b/2)$, noting that $(a+b) \cdot B(a+1,b) = a \cdot B(a,b) > a \cdot B(a,1)=1$, we have $P_{a,b}>0$. Similarly we see that $Q_a,R_{a,b,s}>0$.
%Original:
%\begin{align*}
%K_{a,b}&=\frac{1-\frac{\Gamma(a+b)}{a\Gamma(a)\Gamma(b)}}{b\Gamma(a+b)\sqrt{\frac{1}{(b-2a)\Gamma(b)}(\frac{1}{b\Gamma(2b)}-\frac{1}{b^2\Gamma(b)^2})}}, \\
%D_a&=\frac{1-\frac{\Gamma(3a)}{a\Gamma(a)\Gamma(2a)}}{2a\Gamma(3a)\sqrt{\frac{1}{2a\Gamma(2a)\Gamma(4a)}(1-\frac{\Gamma(4a)}{2a\Gamma(2a)^2})}},\\
%C_{a,b}&=\frac{1-\frac{\Gamma(a+b)}{a\Gamma(a)\Gamma(b)}}{b\Gamma(a+b)\sqrt{(\frac{1}{(2a-b)\Gamma(2a)}-\frac{(2s-1)^2}{a^2\Gamma(a)^2})(\frac{1}{b\Gamma(2b)}-\frac{1}{b^2\Gamma(b)^2})}}.
%\end{align*}

\section{Proof of Theorem \ref{thm:Taniguchi23ThmB}}

\label{sec:proofsB2}

\begin{lemma} \label{lem:Sn2Sigman}
For $b \in (0,1)$,
\begin{align*}
\Cov[S_n^2,\Sigma_n]
&\sim
\begin{cases}
\dfrac{b \cdot B(b,b)-1}{b(b-2a)\Gamma(b)^2}\cdot n^{2b}
%Original:
%\dfrac{1}{(b-2a)\Gamma(2b)}\cdot \left\{ 1-\dfrac{\Gamma(2b)}{b\Gamma(b)^2}\right\} \cdot n^{2b}
&\mbox{if $a \in [-b,b/2)$,}\\
\dfrac{b \cdot B(b,b)-1}{b\Gamma(b)^2}\cdot n^{2b}\log n
%Original:
%\dfrac{1}{\Gamma(2b)} \cdot \left\{ 1-\dfrac{\Gamma(2b)}{b\Gamma(b)^2}\right\} \cdot n^{2b}\log n
&\mbox{if $a=b/2$,} \\
%\dfrac{1}{\Gamma(4a)} \cdot \left\{ 1-\dfrac{\Gamma(4a)}{2a\Gamma(2a)^2}\right\} \cdot n^{4a}\log n
\dfrac{2a \cdot B(2a,b)-1}{b(2a-b)\Gamma(2a)\Gamma(b)}\cdot n^{2a+b}
%Original:
%\dfrac{2a}{(2a-b)b\Gamma(2a+b)} \cdot \left\{ 1-\dfrac{\Gamma(2a+b)}{2a\Gamma(2a)\Gamma(b)} \right\} \cdot n^{2a+b}
&\mbox{if $a \in (b/2,b]$.} \\
\end{cases}
%\label{eq::Taniguchi23ThmA_3}
\end{align*}
\end{lemma}

\begin{proof} First we show that
\begin{align}
E[S_n^2\Sigma_n]\sim
\begin{cases}
\dfrac{1}{(b-2a)\Gamma(2b)} \cdot n^{2b} &\mbox{if $a \in [-b,b/2)$,}\\
\dfrac{1}{\Gamma(2b)} \cdot n^{2b}\log n&\mbox{if $a=b/2$,} \\
%\dfrac{n^{4a}}{\Gamma(4a)} \cdot \log n
\dfrac{2a}{b(2a-b)\Gamma(2a+b)} \cdot n^{2a+b}&\mbox{if $a\in (b/2,b]$.}
\end{cases}
\label{eq:Bercu22JSP(C.23)Asymp}
\end{align}
Since
\begin{align*}
E[S_{n+1}^2\Sigma_{n+1} \mid \mathcal{F}_n]&=E[(S_n+X_{n+1})^2(\Sigma_n+X_{n+1}^2) \mid \mathcal{F}_n]\\
&=\left(1+\dfrac{2a+b}{n}\right) \cdot S_n^2\Sigma_n+\dfrac{2a}{n} \cdot S_n^2+\dfrac{b}{n} \cdot \Sigma_n(\Sigma_n+1),
\end{align*}
we have
\begin{align*}
E[S_{n+1}^2\Sigma_{n+1}]
%&=\left(1+\frac{2a+b}{n}\right)E[S_n^2\Sigma_n]+\frac{2a}{n}E[S_n^2]+\frac{b}{n}(E[\Sigma_n^2]+E[\Sigma_n])\\
&=\left(1+\frac{2a+b}{n}\right) \cdot E[S_n^2\Sigma_n]+\dfrac{2a}{n} \cdot E[S_n^2]+\dfrac{b}{n} \cdot E[\Sigma_n(\Sigma_n+1)].
\end{align*}
Setting $J_n := S_n^2\Sigma_n / c_n(2a+b)$, we have
\begin{align*}
E[J_{n+1}] - E[J_n] = \dfrac{2a \cdot E[S_n^2]}{n \cdot c_{n+1}(2a+b)} + \dfrac{b\cdot E[\Sigma_n(\Sigma_n+1)]}{n\cdot c_{n+1}(2a+b)}, 
\end{align*}
and
\begin{align}
E[J_n] = E[J_1] + \sum_{k=1}^{n-1} \dfrac{2a \cdot E[S_k^2]}{k \cdot c_{k+1}(2a+b)}+ \sum_{k=1}^{n-1}  \dfrac{b\cdot E[\Sigma_k(\Sigma_k+1)]}{k\cdot c_{k+1}(2a+b)}.
\label{eq:Bercu22JSP(C.23)MoreGeneral}
\end{align}
From the definition of $X_1$ %\eqref{def:ERWfirststep} 
and \eqref{eq:ERWSDefSigma}, $E[J_1] =1$. 
It is straightforward to prove that
\begin{align}
\sum_{k=1}^{n-1} \dfrac{c_k(x)}{k \cdot c_{k+1} (y)} 
&= \dfrac{1}{x-y} \cdot \left\{ \dfrac{c_n(x)}{c_n(y)} - 1 \right\} \quad \mbox{for $x,y >-1$ with  $x \neq y$.}
\label{eq:AkimotoTypeFormula}
\end{align}
For the case $a \neq b/2$, using \eqref{eq:Bercu22(A.6)}, \eqref{eq:Bercu22(4.20)1st2nd}, \eqref{eq:Bercu22JSP(C.23)MoreGeneral}, and \eqref{eq:AkimotoTypeFormula}, we see that
%230318
%An explicit formula for the case $a \neq b/2$ is given in Eq. , which can be read as 
\begin{align}
 &E[S_n^2 \Sigma_n] \notag \\
&= \dfrac{1}{b-2a} \left\{ 2b \cdot c_n(2b)-b\cdot c_n(b) + \dfrac{2a}{b} \left( 2a \cdot c_n(2a) - (2a+b) \cdot c_n(2a+b)\right) \right\} ,
% &=\dfrac{2a(2a+b)}{b(2a-b)} \cdot c_n(2a+b)  
% - \dfrac{4a^2}{b(2a-b)} \cdot c_n(2a) \notag \\
% &\quad - \dfrac{2b}{2a-b} \cdot c_n(2b) 
% + \dfrac{b}{2a-b} \cdot c_n(b),
 \label{eq:Bercu22JSP(C.23)}
\end{align}
which coincides with Eq. (C.23) in \cite{Bercu22}.
%(the detail of the calculation is omitted in \cite{Bercu22}).
In view of \eqref{def+asymp:c_n(x)}, the first (resp. fourth) term in the right hand side of \eqref{eq:Bercu22JSP(C.23)} is dominant when $a \in [-b,b/2)$ (resp. $a \in (b/2,b]$). This completes the proof of \eqref{eq:Bercu22JSP(C.23)Asymp} for the case $a \neq b/2$. For the critical case $a=b/2$, by \eqref{def+asymp:c_n(x)}, \eqref{eq:Bercu22JSP(A.6)VarAsymp}, and \eqref{eq:Bercu22(4.20)1st2nd},
%\begin{align*}
%E[S_k^2]  
%\sim \dfrac{1}{\Gamma(b)} \cdot k^b\log k \quad \mbox{and} \quad E[\Sigma_k(\Sigma_k+1)] =2 c_k(2b) \sim \dfrac{2k^{2b}}{\Gamma(2b+1)} 
%\end{align*}
%as $k \to \infty$.
%Thus 
in the right hand side of \eqref{eq:Bercu22JSP(C.23)MoreGeneral}, the third term is dominant. Since $2a=b$, 
%by \eqref{eq:AkimotoTypeFormulaAsymp},
we have
\begin{align*}
E[S_n^2\Sigma_n] &\sim c_n(2b) \cdot \sum_{k=1}^{n-1}  \dfrac{2b\cdot c_k(2b)}{k\cdot c_{k+1}(2b)} %\\
%&
\sim %\dfrac{n^{2b}}{\Gamma(2b+1)} \cdot 2b \cdot \log n = 
\dfrac{1}{\Gamma(2b)} \cdot n^{2b} \log n\quad \mbox{as $n \to \infty$,}
\end{align*}
which is \eqref{eq:Bercu22JSP(C.23)Asymp} for $a=b/2$.
Finally, by \eqref{def+asymp:c_n(x)}, \eqref{eq:Bercu22JSP(4.20)m=1}, and \eqref{eq:Bercu22JSP(A.6)VarAsymp}, we have
\begin{align}
E[S_n^2]E[\Sigma_n] \sim \begin{cases}
\dfrac{1}{b(b-2a)\Gamma(b)^2} \cdot n^{2b} &\mbox{if $a\in[-b,b/2)$,} \\
\dfrac{1}{b\Gamma(b)^2} \cdot n^{2b}\log n
%\dfrac{1}{b\Gamma(2a)\Gamma(b)} \cdot n^{2a+b}\log n 
&\mbox{if $a=b/2$.} \\
\dfrac{1}{b(2a-b)\Gamma(2a)\Gamma(b)} \cdot n^{2a+b}&\mbox{if $a \in (b/2,b]$.}
\end{cases}
\label{eq:Sn2SigmanAnother}
\end{align}
The conclusion of Lemma \ref{lem:Sn2Sigman} follows from \eqref{eq:Bercu22JSP(C.23)Asymp} and \eqref{eq:Sn2SigmanAnother}.
\end{proof}

\begin{lemma} \label{lem:Sn4Var} If $b \in (0,1]$ then as $n \to \infty$, 
\begin{align*}
&V[S_n^2] \\
&\sim
\begin{cases}
\dfrac{3b \cdot B(b,b)-1}{(b-2a)^2\Gamma(b)^2} \cdot n^{2b}
%Original:
%\dfrac{1}{(b-2a)^2}\cdot \left\{ \dfrac{3b}{\Gamma(2b)}-\dfrac{1}{\Gamma(b)^2}\right\} \cdot n^{2b}
&\mbox{if $a \in [-b,b/2)$,}\\
\dfrac{3b \cdot B(b,b)-1}{\Gamma(b)^2} \cdot n^{2b}(\log n)^2
%Original:
%\left\{ \dfrac{6a}{\Gamma(4a)}-\dfrac{1}{\Gamma(2a)^2}\right\} \cdot n^{4a}(\log n)^2
&\mbox{if $a=b/2$,}\\
\dfrac{1}{(2a-b)^2 \Gamma(2a)^2} \cdot \left\{ \dfrac{6(2a^2+2ab-b^2)}{4a-b} \cdot B(2a,2a) -1 \right\} \cdot n^{4a}
%Original:
%\dfrac{1}{(2a-b)^2} \cdot \left\{ \dfrac{6(2a^2+2ab-b^2)}{(4a-b)\Gamma(4a)}-\dfrac{1}{\Gamma(2a)^2}\right\} \cdot n^{4a}
&\mbox{if $a \in (b/2,b]$.} \\
\end{cases}
\end{align*}
\end{lemma} 

\begin{proof} First we prove
\begin{align}
E[S_n^4] \sim
\begin{cases}
\dfrac{3b}{(b-2a)^2\Gamma(2b)}\cdot n^{2b}&\mbox{if $a \in [-b,b/2)$,}\\
\dfrac{3b}{\Gamma(2b)} \cdot n^{2b}(\log n)^2&\mbox{if $a=b/2$,}\\
%\dfrac{6a}{\Gamma(4a)} \cdot n^{4a}(\log n)^2
\dfrac{6(2a^2+2ab-b^2)}{(2a-b)^2(4a-b)\Gamma(4a)}\cdot n^{4a}&\mbox{if $a \in (b/2,b]$.}\\
\end{cases}
\label{eq:Bercu22JSP(C.24)Asymp}
\end{align}
%Using Bercu's result:
%From Eq. (3.23) in \cite{Bercu22}, if $2a>b$, then
%\begin{align*}
%\lim_{n \to \infty} E\left[ \left(\dfrac{S_n}{n^a}\right)^4\right] = \dfrac{6(2a^2+2ab-b^2)}{(2a-b)^2(4a-b)\Gamma(4a)}.
%\label{eq:Bercu22JSP(3.23)}
%\end{align*}
%
%A similar method can be applied to $E[S_n^4]$: 
Our starting point is Eq. (C.22) in \cite{Bercu22}, i.e.
\begin{align*}
E[S_{n+1}^4]=\left(1+\frac{4a}{n}\right)\cdot E[S_n^4]+\dfrac{6b}{n}\cdot E[S_n^2\Sigma_n]+\dfrac{4a}{n}\cdot E[S_n^2]+\dfrac{b}{n} \cdot E[\Sigma_n].
\end{align*}
Putting $L_n := S_n^4/ c_n(4a)$, we see that
\begin{align}
E[L_n] = 1 + \sum_{k=1}^{n-1} \dfrac{6b \cdot E[S_k^2\Sigma_k]}{k \cdot c_{k+1}(4a)} + \sum_{k=1}^{n-1} \dfrac{4a \cdot E[S_k^2]}{k \cdot c_{k+1}(4a)} + \sum_{k=1}^{n-1} \dfrac{b \cdot E[\Sigma_k]}{k \cdot c_{k+1}(4a)}.
\label{eq:Bercu22JSP(C.24)MoreGeneral}
\end{align}
For the case $a \neq b/2$, using \eqref{eq:Bercu22JSP(C.23)}, \eqref{eq:Bercu22(A.6)}, \eqref{eq:Bercu22JSP(4.20)m=1}, and \eqref{eq:AkimotoTypeFormula}, we have
%230318
%we can see that if $a \neq b/2$ then
\begin{align}
E[S_n^4] &= \dfrac{24a(2a^2+2ab-b^2)}{(2a-b)^2(4a-b)} \cdot c_n(4a) 
-\dfrac{12a(2a+b)}{(2a-b)^2} \cdot c_n(2a+b) \notag \\
&\quad +\dfrac{8a}{2a-b} \cdot c_n(2a)
+ \dfrac{6b^2}{(2a-b)^2} \cdot c_n(2b)- \dfrac{b(5b-2a)}{(2a-b)(4a-b)} \cdot c_n(b), 
\label{eq:Bercu22JSP(C.24)}
\end{align}
which coincides with Eq. (C.24) in \cite{Bercu22}.
By \eqref{def+asymp:c_n(x)}, the fourth (resp. first) term in the right hand side of \eqref{eq:Bercu22JSP(C.24)} is dominant when $a \in [-b,b/2)$ (resp. $a \in (b/2,b]$). This completes the proof of \eqref{eq:Bercu22JSP(C.24)Asymp} for the case $a \neq b/2$. For the critical case $a=b/2$, from \eqref{def+asymp:c_n(x)}, \eqref{eq:Bercu22JSP(C.23)Asymp}, \eqref{eq:Bercu22JSP(A.6)VarAsymp}, and \eqref{eq:Bercu22(4.20)1st2nd}, the first term is dominant in the right hand side of \eqref{eq:Bercu22JSP(C.24)MoreGeneral}. Noting that $4a=2b$, we have
\begin{align*}
E[S_n^4] &\sim c_n(2b) \cdot \sum_{k=1}^{n-1}  \dfrac{6b\cdot E[S_k^2\Sigma_k]}{k\cdot c_{k+1}(2b)} 
%&\sim \dfrac{n^{2b}}{\Gamma(2b+1)} \cdot  \sum_{k=1}^{n-1}  \dfrac{6b\Gamma(2b+1) \log k }{\Gamma(2b) k} 
\sim \dfrac{6b n^{2b}}{\Gamma(2b)} \cdot  \sum_{k=1}^{n-1}  \dfrac{ \log k }{ k} 
\sim \dfrac{3b}{\Gamma(2b)} \cdot  n^{2b} (\log n)^2 %\quad \mbox{as $n \to \infty$,}
\end{align*}
as $n \to \infty$, which is \eqref{eq:Bercu22JSP(C.24)Asymp} for $a=b/2$.
%By \eqref{eq:Bercu22JSP(A.6)VarAsymp}, 
%\begin{align}
%E[S_n^2]^2 \sim \begin{cases}
%\dfrac{1}{(b-2a)^2\Gamma(b)^2} \cdot n^{2b} &\mbox{if $a<b/2$,} \\
%\dfrac{1}{\Gamma(b)^2} \cdot n^{2b}(\log n)^2 &\mbox{if $a=b/2$.} \\
%\dfrac{1}{\Gamma(2a)^2} \cdot n^{4a}(\log n)^2
%\dfrac{1}{(2a-b)^2\Gamma(2a)^2} \cdot n^{4a}&\mbox{if $a>b/2$.}
%\end{cases}
%\label{eq:Sn4Another}
%\end{align}
Lemma \ref{lem:Sn4Var} follows from \eqref{eq:Bercu22JSP(C.24)Asymp} and \eqref{eq:Bercu22JSP(A.6)VarAsymp}.
%\eqref{eq:Sn4Another}.
\end{proof}

Theorem \ref{thm:Taniguchi23ThmB} follows from Lemmata \ref{lem:Sn2Sigman}, \ref{lem:Sn4Var}, and \ref{lem:VSigman}.

%Original:
%\begin{align*}
%K'_b&=\frac{1-\frac{\Gamma(2b)}{b\Gamma(b)^2}}{\Gamma(2b)\sqrt{(\frac{3b}{\Gamma(2b)}-\frac{1}{\Gamma(b)^2})(\frac{1}{b\Gamma(2b)}-\frac{1}{b^2\Gamma(b)^2})}}, \\
%C'_{a,b}&=\frac{2a(1-\frac{\Gamma(2a+b)}{2a\Gamma(2a)\Gamma(b)})}{b\Gamma(2a+b)\sqrt{(\frac{6(2a^2+2ab-b^2)}{(4a-b)\Gamma(4a)}-\frac{1}{\Gamma(2a)^2})(\frac{1}{b\Gamma(2b)}-\frac{1}{b^2\Gamma(b)^2})}}.
%\end{align*}

%\begin{acknowledgements}
\section*{Acknowledgement}
%If you'd like to thank anyone, place your comments here
%and remove the percent signs.
The authors thank the referee for helpful and constructive comments, in particular on the heuristics behind our results in Section \ref{ss:CorrERWS}. M.T. is partially supported by JSPS KAKENHI Grant Numbers JP19H01793, JP19K03514 and JP22K03333.
%\end{acknowledgements}
%\appendix

%% The Appendices part is started with the command \appendix;
%% appendix sections are then done as normal sections
%% \appendix

%% \section{}
%% \label{}

%% If you have bibdatabase file and want bibtex to generate the
%% bibitems, please use
%%
%%  \bibliographystyle{elsarticle-num} 
%%  \bibliography{<your bibdatabase>}

%% else use the following coding to input the bibitems directly in the
%% TeX file.

% Non-BibTeX users please use

\end{document}